\documentclass{amsart}
\usepackage{amssymb,amsmath,latexsym}
\usepackage{amsthm}
\usepackage{fontenc}
\usepackage{amssymb}
\numberwithin{equation}{section}
\newtheorem{theorem}{Theorem}[section]

\DeclareMathOperator*{\sgn}{sgn}
\DeclareMathOperator*{\sinc}{sinc}

\begin{document}
\author{Alexander E. Patkowski}
\title{An application of Titchmarsh's theorem and the Salem equivalence}

\maketitle
\begin{abstract}We extend the equivalence of the Salem type for the Riemann hypothesis by application of Titchmarsh's theorem. Other equivalences to the Riemann hypothesis and notes on related Fourier integrals are provided. \end{abstract}

% AMS keywords (used in AMS journals)

% AMS subject classifications (used in AMS journals)

\section{Introduction}
The Riemann hypothesis is the claim that all the non-trivial zeros of $\zeta(s)$ are on the line $\Re(s)=\frac{1}{2},$ where $\zeta(s)=\sum_{n\ge1}n^{-s}$ is the Riemman zeta function [7].  The Salem equivalence of the Riemann hypothesis [5] states that the Riemann hypothesis can be true only if
\begin{equation} \int_{-\infty}^{\infty}\frac{e^{-\sigma x}f(x)}{e^{e^{y-x}}+1}dx=0,\end{equation} has no bounded solution other than $f(y)\equiv0,$ when $\frac{1}{2}<\sigma<1,$ $\sigma\in\mathbb{R}.$ This can be seen from first noting that (1.1) is in the form of a convolution (after multiplying by $e^{\sigma y}$),
$$(f\star g)(y):=\int_{-\infty}^{\infty}f(x)g(y-x)dx.$$ Salem then carefully selects a Mellin transform (see [6, pg.46, Theorem 28]) for the Riemann zeta function, which is analytic in the critical region $1>\Re(s)=\sigma>0.$ Namely, for $\Re(s)>0$ it is well-known [7, pg.21, Theorem 2.5] that
\begin{equation}\int_{0}^{\infty}\frac{x^{s-1}dx}{e^{x}+1}=\Gamma(s)\zeta(s)(1-2^{1-s}), \end{equation} 
where $\Gamma(s)$ is the classical gamma function which has no zeros.
A change of variable in (1.2), gives the integral
\begin{equation} \int_{-\infty}^{\infty}\frac{e^{(\sigma+it) x}}{e^{e^{x}}+1}dx=\Gamma(\sigma+it)\zeta(\sigma+it)(1-2^{1-\sigma-it}). \end{equation} Taking the Fourier transform of (1.1) , we see from (1.3) that if the Fourier transform [6] of $f(x)$ is given by
$$\mathfrak{F}f(t):=\frac{1}{\sqrt{2\pi}}\int_{-\infty}^{\infty}e^{-itx}f(x)dx,$$
then
$$\Gamma(\sigma-it)\zeta(\sigma-it)(1-2^{1-\sigma+it})\mathfrak{F}f(t)=0$$ when $\frac{1}{2}<\sigma<1,$ can only be true if $\mathfrak{F}f(t)=0$ should $\zeta(\sigma+it)$ have no zeros off the line $\Re(s)=\frac{1}{2}.$ Now [3, pg.419, Corollary 8.5] tells us that if $\mathfrak{F}f(t)=0$ then $f(x)$ is zero almost everywhere.
\par 
Throughout, we will use standard notation [6, pg.10] for $f\in L^{p}(a,b)$ to mean 
$$||f||_p:=\left(\int_{a}^{b}|f(x)|^{p}dx\right)^{1/p}<\infty.$$ We will also make regular use of Hilbert transforms, which we define to be [6, pg.120, eq.(5.1.11)]
$$\mathfrak{H}f(y):=\frac{1}{\pi}PV\int_{-\infty}^{\infty}\frac{f(x)}{x-y}dx,$$
for $f\in L^{2}(-\infty,\infty)$ where $PV$ means the principal value at $x=y.$ Moreover, [6, pg.120, eq.(5.1.11)] says that applying the Hilbert operator twice gives $\mathfrak{H}\mathfrak{H}f(y)=-f(y).$ As noted in [3, pg.435], the Fourier transform of $\mathfrak{H}f(x)$ is equal to $-i\sgn(t)\mathfrak{F}f(t).$ Here $\sgn(t)$ gives a positive sign when $t>0,$ a negative sign when $t<0,$ and $0$ otherwise.

\par
Thus far, it seems little has been done in the way of finding new routes of Salem's equivalence (1.1) toward the truth of the Riemann Hypothesis. Many papers (e.g. [2, 9]) have noted Wiener's work [8, pg.9, Theorem II], [4, pg.32], relating the kernel from (1.1) to the linear span of its translates being dense in $L^{1}(-\infty,\infty)$. The purpose of this paper is to extend the Riemann Hypothesis equivalence of the Salem type to other criteria by applying Titchmarsh's theorem, which we now state.
\begin{theorem}(Titchmarsh [6, pg.128, Theorem 95]) Alternative necessary and sufficient conditions that a complex-valued $h(x)\in L^{2}(-\infty,\infty)$ is the limit as $z\rightarrow x$ of a function $h(z)$ analytic in the upper half-plane $\Im(z)>0,$ such that
$$\int_{-\infty}^{\infty}|h(x+iy)|^2dx<C<\infty,$$
are (i) $h(x)=f(x)-ig(x)$ where $f$ and $g$ are Hilbert transforms of each other, \newline
(ii) The Fourier transform $\mathfrak{F}h(t)$ is null for $t<0.$
\end{theorem}
As it stands, it would seem to be a simple matter to apply (1.1)--(1.3) to criteria (ii) of Theorem 1.1. However, the condition $t<0$ adds some difficulty since the Riemann hypothesis requires us to consider $t\in\mathbb{R},$ since zeros are either on the line $\Re(s)=\frac{1}{2},$ or are symmetrical pairs about this line [7, pg.30]. To address this issue, we will create a new kernel which satisfies a Salem type equivalence, but which covers all possible symmetrical pairs when $t<0,$ $\frac{1}{2}<\sigma<1.$ The kernel we constructed to achieve this is given as $x>0,$
\begin{equation}k(x):=\sum_{n=1}^{\infty}d_{1/2}(n)e^{-nx}-x^{-3/2}\sqrt{\frac{\pi}{2}}\zeta\left(\frac{3}{2}\right)-x^{-1}\zeta\left(\frac{1}{2}\right).\end{equation}
where [7, pg.8] $d_r(n)=\sum_{d|n}d^{r}.$ Define the convolution integral $I_{\sigma}(z)$ to be for $f\in L^1(-\infty,\infty),$
$$I_{\sigma}(z):=(f(x)\star e^{\sigma x}k(e^{x}))(z)=\int_{-\infty}^{\infty}f(x)e^{\sigma(z-x)}k(e^{z-x})dx.$$ To apply Theorem 1.1, we will also need to construct a complex valued function. Since $k:\mathbb{R}\rightarrow\mathbb{R},$ and for real valued $f,$ $I_{\sigma}:\mathbb{R}\rightarrow\mathbb{R},$ we multiply by a complex exponential to define $\bar{I}_{\sigma,m}(z):=e^{-imz}I_{\sigma}(z)$ for a positive real number $m\in\mathbb{R}^{+}.$ The new function has the property $\bar{I}_{\sigma,m}:\mathbb{R}\rightarrow\mathbb{C},$ where $f(x)$ is chosen to be real-valued. To illustrate an example of this procedure applicable to Theorem 1.1, put
$$ \overline{\sinc}(x):=e^{-imx}\frac{\sin(x)}{x}=e^{-imx}\sinc(x),$$
where $\sinc(x)\in L^2(-\infty,\infty),$ and $\sinc:\mathbb{R}\rightarrow\mathbb{R},$ it follows that $\overline{\sinc}:\mathbb{R}\rightarrow\mathbb{C}.$ From [1, pg.252, eq.(44)], we have the known Hilbert transform
\begin{equation}\mathfrak{H}\frac{\sin(mx)}{x}=\frac{\cos(mx)-1}{x}. \end{equation}
The Hilbert transform of the real part is then computed to be
$$\begin{aligned}\mathfrak{H}\Re\left(\frac{e^{-imx}\sin(x)}{x}\right)&=\mathfrak{H}\frac{\cos(mx)\sin(x)}{x}\\
&= \mathfrak{H} \frac{1}{2x}\left(\sin((m+1)x)-\sin((m-1)x) \right) \\
&=\frac{1}{2x}\left(\cos((m+1)x)-1-\cos((m-1)x)+1)\right)\\
&=\frac{1}{x}\sin(mx)\sin(x) \\
&=\Im\left( \frac{e^{-imx}\sin(x)}{x}\right),\end{aligned}$$
by $\cos(x)-\cos(y)=-2\sin((x+y)/2)\sin((x-y)/2)$ and similar identities. Additionally, we have $$\mathfrak{F}e^{-imx}\frac{\sin(x)}{x}=0,$$ when $x<-m-1.$
\par
Now we may state our main theorem by using a modified form of Theorem 1.1 also given by Titchmarsh [6, pg.129, Theorem 96], which we will utilize in the following section. Recall the standard notation $f=O(g)$ to mean that $|f(x)|\le C g(x)$ for some constant $C>0.$
\begin{theorem}\label{thm:theorem 1.2} Let $\frac{1}{2}<\sigma<1.$ Alternative necessary and sufficient conditions that a complex-valued $\bar{I}_{\sigma,m}(x)\in L^{2}(-\infty,\infty)$ is the limit as $z\rightarrow x$ of a function $\bar{I}_{\sigma,m}(z)$ analytic in the upper half-plane $\Im(z)>0,$ such that
$$\int_{-\infty}^{\infty}|\bar{I}_{\sigma,m}(x+iy)|^2dx=O(e^{2my}),$$
are (i) $\bar{I}_{\sigma,m}(x)=v(x)-iw(x)$ where $v$ and $w$ are Hilbert transforms of each other, \newline
(ii) The Fourier transform $\mathfrak{F}\bar{I}_{\sigma,m}(t)$ is null for $t<-m.$ \newline
(iii) $\mathfrak{F}f(t)=0$ for $t<0,$ if the Riemann hypothesis is true.
\end{theorem}

\section{Proof of the Theorem}
To obtain our theorem, we will need to derive $k$ and establish its properties related to $L^2(-\infty,\infty),$ in order to apply the initial statement of the theorem as well as Hilbert transforms. A similar function to $k$ was considered in [7, pg.160]. The modified form of Theorem 1.1 [6, Theorem 96], applies when replacing a complex-valued $h(z)$ by $e^{-imz}h(z).$ However, note in our case we are multiplying $e^{-imz}$ by a real-valued function to obtain a complex-valued function. This is permissible provided the function is $L^2(-\infty,\infty),$ and the resulting conditions are mostly the same.
\begin{proof}[Proof of Theorem~\ref{thm:theorem 1.2}]
As noted in the introduction, the functional equation [7, pg.16, eq.(2.1.13)] says the critical zeros are either on the line $\Re(s)=\frac{1}{2},$ or appear in pairs about the line $\Re(s)=\frac{1}{2}.$ If we set $\frac{1}{2}<\sigma<1,$ and if $\zeta(\sigma-it)\zeta(\sigma-\frac{1}{2}-it)\neq0$ for $t<0,$ we also have $\zeta(1-\sigma+it)\zeta(\frac{3}{2}-\sigma+it)\neq0$ for $t<0$ by [7, pg.16, eq.(2.1.13)]. This is equivalent to $\zeta(\sigma-it)\neq0$ for $\sigma\neq\frac{1}{2},$ $\sigma\in(0,1),$ $t\in\mathbb{R}.$

We construct the kernel $k$ by noting by the Mellin inversion formula for $\Re(s)=c>\Re(r)+1,$ $x>0,$
\begin{equation} \sum_{n=1}^{\infty}d_{r}(n)e^{-nx}=\frac{1}{2\pi i}\int_{(c)}\Gamma(s)\zeta(s)\zeta(s-r)x^{-s}ds,\end{equation}
by [7, pg.8, eq.(1.3.1)]. Since we wish to obtain a Mellin transform which is analytic in the critical region for $\zeta(s),$ we set $r=\frac{1}{2}$ and move the contour to the left to $\Re(s)=d,$ $0<d<1.$ Computing the residue at the pole $s=\frac{3}{2}$ from $\zeta(s-\frac{1}{2}),$ and the pole $s=1$ from $\zeta(s),$ we get 
\begin{equation} k(x)=\frac{1}{2\pi i}\int_{(d)}\Gamma(s)\zeta(s)\zeta\left(s-\frac{1}{2}\right)x^{-s}ds,\end{equation}
which is easily seen to be equivalent to
$$\int_{-\infty}^{\infty}e^{(\sigma+it)x}k(e^{x})dx=\Gamma(\sigma+it)\zeta(\sigma+it)\zeta\left(\sigma+it-\frac{1}{2}\right).$$
Plancherel's theorem [3, pg.419, Theorem 8.6] for Fourier transforms tell us that 
$$\int_{-\infty}^{\infty}|e^{\sigma x}k(e^{x})|^2dx=\int_{-\infty}^{\infty}\bigg|\Gamma(\sigma-ix)\zeta(\sigma-ix)\zeta\left(\sigma-ix-\frac{1}{2}\right)\bigg|^2dx.$$
The integral on the right side is bounded by Stirling's formula [7, pg.160] 
$$|\Gamma(\sigma+it)|=e^{-\frac{\pi}{2}|t|}|t|^{\sigma-\frac{1}{2}}\sqrt{2\pi}\big(1+O(t^{-1})\big),$$
and $\zeta(s)$ having finite order in the critical strip [7, pg.95]. Therefore, the isometric property of the Fourier transform [3, pg.421, Theorem 8.9] on $L^2(-\infty,\infty)$ tells us that $e^{\sigma x}k(e^{x})\in L^2(-\infty,\infty).$ Now Minkowski's inequality for integrals [3, pg.323, Theorem 5.60] tells us that [3, pg.346, Proposition 6.14, (b)],
$$||f\star e^{\sigma x}k(e^{x})||_2\le ||f||_1||e^{\sigma x}k(e^{x})||_2.$$
Hence if we assume $f\in L^1(-\infty,\infty),$ then $f\star e^{\sigma x}k(e^{x})\in L^2(-\infty,\infty).$
\par If we assume the initial criteria that $\bar{I}_{\sigma,m}(z)$ is analytic for $\Im(z)>0,$ then we have the Cauchy integral representation by [6, Theorem 93], [6, pg.123, eq.(5.3.4)--(5.3.6)],
$$\begin{aligned}\bar{I}_{\sigma,m}(x+iy)&=\frac{1}{i\pi}\int_{-\infty}^{\infty}\frac{\bar{I}_{\sigma,m}(t)}{t-(x+iy)}dt\\
&=\bar{I}_{\sigma,m}(x)*P_{y}+i\bar{I}_{\sigma,m}(x)*Q_{y}.\end{aligned}$$
Here we have used the Poisson kernel,
$$P_{y}(x)=\frac{1}{\pi}\frac{y}{y^2+x^2},$$
and its conjugate kernel 
$$Q_{y}(x)=\frac{1}{\pi}\frac{x}{y^2+x^2}.$$
The condition $(i)$ follows from letting $y\rightarrow0$ to get $\bar{I}_{\sigma,m}(x)*P_{y}\rightarrow v(x),$ and $\bar{I}_{\sigma,m}(x)*Q_{y}\rightarrow -w(x),$ and then comparing $w(x)$ with the definition of the Hilbert transform.

To see how condition (ii) follows from (i), observe that by (ii) if $x<-m,$ 
$$\begin{aligned}\mathfrak{F}\bar{I}_{\sigma,m}(x)&=\mathfrak{F}v(-x)-i\mathfrak{F}\mathfrak{H}v(-x)\\
&=0,\end{aligned}$$
by [6, pg.120, eq.(5.1.8)] and $\mathfrak{F}\mathfrak{H}v(-x)=-i\sgn (-x)\mathfrak{F}v(-x).$ The converse follows from the formulae for $v$ and $w$ at the end of the page [6, pg.119], which we leave for the reader. Note that Parseval's theorem [6, pg.127] tells us that
$$\int_{-\infty}^{\infty}|\mathfrak{F}\bar{I}_{\sigma,m}(x)|^2e^{-2xy}dx=\int_{-\infty}^{\infty}|\bar{I}_{\sigma,m}(x+iy)|^2dx.$$
And since we must have
$$e^{m2y}\int_{-\infty}^{-m}|\mathfrak{F}\bar{I}_{\sigma,m}(x)|^2dx\le \int_{-\infty}^{-m}|\mathfrak{F}\bar{I}_{\sigma,m}(x)|^2e^{-2xy}dx\le C,$$
as $y\rightarrow\infty,$ it follows that $\mathfrak{F}\bar{I}_{\sigma,m}(x)=0$ for $x<-m.$ On the other hand, this implies
$$\int_{-\infty}^{\infty}|\mathfrak{F}\bar{I}_{\sigma,m}(x)|^2e^{-2xy}dx\le e^{m2y}\int_{-m}^{\infty}|\mathfrak{F}\bar{I}_{\sigma,m}(x)|^2dx\le Ce^{2my},$$
and the upper bound follows in the initial statement from (ii). For the equivalence between (iii) and (ii), observe that 
$$\mathfrak{F}\bar{I}_{\sigma,m}(t)=\Gamma(\sigma-i(t+m))\zeta(\sigma-i(t+m))\zeta\left(\sigma-i(t+m)-\frac{1}{2}\right)\mathfrak{F}f(t+m).$$ The right side is zero in the region $\frac{1}{2}<\sigma<1,$ $t<-m$ if the Riemann zeta function has no zeros off the line $\Re(s)=\frac{1}{2}$ and $\mathfrak{F}f(t)=0$ for $t<0.$ Otherwise, the Riemann Hypothesis would be false.\end{proof}

\begin{section}{Related observations}
Another result that is deduced from [6, Theorem 96], is obtained from the upper bound in Theorem 1.2. Namely, by taking logarithms of both sides, and setting $m$ to be the smallest number for which $\mathfrak{F}\bar{I}_{\sigma,m}(x)=0$ such that $x<-m,$ then we have
\begin{equation}\lim_{y\rightarrow\infty}\frac{1}{y}\log\int_{-\infty}^{\infty}|\bar{I}_{\sigma,m}(x+iy)|^2dx=2m.\end{equation}
A condition for the Riemann hypothesis would be that (3.1) implies $m$ is the smallest number for which $\mathfrak{F}f(x+m)=0,$ such that $x<-m,$ or simply that $\mathfrak{F}f(x)=0,$ such that $x<0.$
\par It is also possible to use our $I_{\sigma}(z)$ function to apply [3, pg.419, Theorem 8.6] to obtain the Salem type equivalence. Note the one-one property was not applied in Theorem 1.2 to relate $f$ directly to the Riemann hypothesis due to the condition $t<-m$ on its Fourier transform. However, the Riemann hypothesis would be true should no bounded measurable function $f$ satisfy 
$$(f(x)\star e^{\sigma x}k(e^{x}))(z)=\int_{-\infty}^{\infty}f(x)e^{\sigma(z-x)}k(e^{z-x})dx=0,$$
other than the trivial case $f\equiv0.$ Moreover, by [8, pg.9, Theorem II], the set of all translations of $e^{\sigma x}k(e^{x})$ is closed in $L^1(-\infty,\infty)$ if the Riemann hypothesis is true.
\par We next consider results of Paley and Wiener [4, pg.16, Theorem XII] and [4, pg.24] by first writing
$$\bar{I}_{\sigma,m}(-t)=\int_{-\infty}^{\infty}e^{-ixt}\Gamma(\sigma-i(x+m))\zeta(\sigma-i(x+m))\zeta\left(\sigma-i(x+m)-\frac{1}{2}\right)\mathfrak{F}f(x+m)dx.$$
Recall that $f$ is real-valued, and consequently if $t$ is real then $|\bar{I}_{\sigma,m}(-t)|=|I_{\sigma,m}(-t)|$ is real and non-negative. If we put 
$$f_1(x):=\Gamma(\sigma-i(x+m))\zeta(\sigma-i(x+m))\zeta\left(\sigma-i(x+m)-\frac{1}{2}\right)\mathfrak{F}f(x+m),$$ then
as before $f_1(x)=0,$ for $x<-m$ implies $\mathfrak{F}f(x)=0$ for $x<0$ if the Riemann hypothesis is true. Define $f_2(x)$ to be zero if $x>m,$ and non-zero otherwise. Then by [4, pg.24], we have the Fourier transform of $f_1f_2,$
\begin{equation}\int_{-m}^{m}e^{-itx}f_1(x)f_2(x)dx=\int_{-\infty}^{\infty}\bar{I}_{\sigma,m}(-x)\mathfrak{F}f_2(t-x)dx,\end{equation}
if
\begin{equation}\int_{-\infty}^{\infty}\frac{|\log|\bar{I}_{\sigma,m}(-x)||}{1+x^2}dx<\infty,\quad\text{and} \int_{-\infty}^{\infty}\frac{|\log|\mathfrak{F}f_2(x)||}{1+x^2}dx<\infty.\end{equation}
\begin{theorem} If the Fourier transform (3.2) implies $\mathfrak{F}f(x)=0$ for $x<0,$ then the Riemann hypothesis is true. A sufficient condition for (3.3) follows if
$$\bar{I}_{\sigma,m}(-x)=O(e^{-|x|^{\epsilon-1}}),$$
for $1<\epsilon<2,$ and similarly for $\mathfrak{F}f_2.$
\end{theorem}
\begin{proof}Since we have already shown how to derive (3.2), we prove the estimate. Notice that [6, pg.181, eq.(7.3.7)],
\begin{equation} \int_{-\infty}^{\infty}\frac{|x|^{\epsilon-1}}{1+x^2}dx=2\int_{0}^{\infty}\frac{x^{\epsilon-1}}{1+x^2}dx,\end{equation}
is twice $\pi/\sin(\epsilon\pi/2)$ for $0<\epsilon<2.$ When comparing with (3.3), and the fact that $\bar{I}_{\sigma,m}(-x)\in L^2(-\infty,\infty),$ we see that it is sufficient to have for $2>\epsilon>1,$
$$\log|\bar{I}_{\sigma,m}(-x)|\le -C|x|^{\epsilon-1},$$ and similar for $\mathfrak{F}f_2.$ Since $e^{-|x|^{\delta}}\in L^2(-\infty,\infty),$ $\delta>0,$ the result follows.
\end{proof}
\end{section}

1390 Bumps River Rd. \\*
Centerville, MA
02632 \\*
USA \\*
ul. A. E. Ody\'{n}ca 47 \\*
02-606 Warsaw\\*
Poland\\*
E-mail: alexpatk@hotmail.com, alexepatkowski@gmail.com

\end{document}